\documentclass[10pt]{paper}

\usepackage{amsfonts,amsmath,amsthm,amssymb,mathrsfs}
\usepackage{fullpage}
\usepackage{enumitem}
\usepackage[latin1]{inputenc}
\usepackage{graphicx,xcolor}
\usepackage{hyperref}
\usepackage{fancyhdr}

\def\G{\mathcal{G}}
\def\A{\mathcal{A}}
\def\B{\mathcal{B}}
\def\C{\mathcal{C}}

\newtheorem{theorem}{Theorem}

\newtheorem{lemma}[theorem]{Lemma}

\AtEndDocument{%
\par
\bigskip
    \noindent\textsc{Jordi Castellv\'i}\\
    Departament de Matem\`atiques de la Universitat Polit\`ecnica de Catalunya (UPC), Barcelona, Spain.\\
    \textit{E-mail}: \texttt{jordi.castellvi@upc.edu}
\par
\bigskip
	\noindent\textsc{Marc Noy}\\
	Departament de Matemàtiques and Institut de Matemàtiques (IMTech) de la Universitat Politècnica de Catalunya (UPC), and Centre de Recerca Matemàtica (CRM), Barcelona, Spain.\\
	\textit{E-mail}: \texttt{marc.noy@upc.edu}
\par
\bigskip
	\noindent\textsc{Cl\'ement Requil\'e}\\
	Departament de Matem\`atiques and Institut de Matemàtiques (IMTech) de la Universitat Polit\`ecnica de Catalunya (UPC), Barcelona, Spain.\\
	\textit{E-mail}: \texttt{clement.requile@upc.edu}
}

\title{Enumeration of chordal planar graphs and maps}

\author{Jordi Castellv\'i \and Marc Noy \and Cl\'ement Requil\'e}

\begin{document}
\maketitle

\begin{abstract}
	We determine the number of labelled chordal planar graphs with $n$ vertices, which is  asymptotically  $c_1\cdot n^{-5/2} \gamma^n n!$ for a constant $c_1>0$ and $\gamma \approx 11.89235$. 
	We also determine the number of rooted simple chordal planar maps with $n$ edges, which is asymptotically $c_2 n^{-3/2} \delta^n$, where  $\delta = 1/\sigma \approx 6.40375$, and $\sigma$ is an algebraic number of degree 12. The proofs are based on combinatorial decompositions and singularity analysis. 
	Chordal planar graphs (or maps) are a natural example of a subcritical class of graphs in which the class of 3-connected graphs is relatively rich. 
	The 3-connected members are precisely chordal triangulations, those obtained starting from $K_4$ by repeatedly adding vertices adjacent to an existing triangular face.
\end{abstract}

%
%
\section{Introduction}

A graph is chordal if every cycle of length greater than three contains at least one chord, which is an edge connecting non-consecutive vertices of the cycle. 
Chordal graphs have been much studied in structural graph theory and graph algorithms (see for instance \cite{G80}), but much less from the point of view of enumeration. 
It is known that the asymptotic number of labelled chordal graphs with $n$ vertices is $\binom{n}{n/2}2^{n^2/4}$; an explanation for this estimate is that as $n$ goes to infinity almost all chordal graphs with $n$ vertices are split, that is, the vertex set can be partitioned into a clique and an independent set \cite{BRW85}. 
See also \cite{W85} for results on the exact counting of chordal labelled graphs. 

On the other hand, there has been much work on counting planar graphs and related classes of graphs since the seminal work by Giménez and Noy \cite{GN09}. 
Here we focus on planar graphs that are at the same time chordal. 
To count them we use, as in \cite{GN09}, the canonical decomposition of graphs into $k$-connected components for $k = 1, 2, 3$. 
The starting point is the enumeration of 3-connected chordal planar graphs: these are precisely the chordal triangulations, which when suitably rooted are in bijection with ternary trees. 
Then we use the decomposition of 2-connected graphs into 3-connected components. 
An important difference with the class of all planar graphs is that one cannot compose more than two graphs in series since otherwise a chordless cycle is created. 
A more significant difference is that the class of chordal planar graphs is \textit{subcritical} (a concept defined below), instead of being \textit{critical} as the class of all planar graphs: this is reflected by the polynomial term $n^{-5/2}$ of the asymptotic estimates for the number of graphs in the class \cite{DFKKR11}, as opposed to $n^{-7/2}$ for all planar graphs \cite{GN09}. 
Thus we have a natural example in which the class of 3-connected graphs is relatively rich, yet the class is subcritical.

Before stating our main results, we recall the formal definition of a subcritical class of graphs, which is based on properties of the associated generating functions. 
A class of graphs $\G$ is \emph{block-stable} if it has the property that a graph belongs to $\G$ if and only if each of its blocks belongs to $\G$.
Let $\G$ be a block-stable class of labelled graphs and denote the subclasses of connected graphs and $2$-connected graphs by $\C$ and $\B$, their associated exponential generating functions by $C(x)$ and $B(x)$ with radius of convergence  $\rho_b$ and $\rho_C$, respectively. 
The class $\G$ is said to be \emph{subcritical} \cite{DFKKR11} if 
\begin{equation}\label{eq:subcritical}
	\rho C'(\rho) < \rho_b.
\end{equation} 
This condition has important implications on the structure of a class of graphs. 
Intuitively, subcritical classes are ``tree-like'' in some sense \cite{GNR13} exhibited for instance by the fact that their scaling limit is the continuum random tree \cite{PSW16}, which means that the global structure is essentially determined by the block-decomposition tree, while the size of the blocks is bounded in expectation and at most logarithmic. 

Our first result is the following.

\begin{theorem}\label{th:graph}
Let $g_n$ be the number of labelled chordal planar graphs with $n$ vertices, $c_n$ those which are connected, and $b_n$ those which are 2-connected. 
Then as $n\to\infty$ we have 
\begin{enumerate}
	\item $g_n \sim g \cdot n^{-5/2} \gamma^n n!,
		\qquad \gamma \approx 11.89235, \ g\approx 0.00027205$
	\item $c_n \sim c \cdot n^{-5/2} \gamma^n n!, 
		\qquad c\approx 0.00027194$,
	\item $b_n \sim b \cdot n^{-5/2} \gamma_b^n n!,
		\qquad \gamma_b\approx 10.76897, \ b\approx 0.00016215$.
\end{enumerate}
\end{theorem}

\noindent
We can add to the previous estimates the formula (see \cite{BP71}) for the number $t_n$ of 3-connected labelled chordal graphs
\begin{equation}\label{eq:tn}
	t_n = \binom{n}{3}\frac{(3n-9)!}{(2n-4)!}\approx t \cdot n^{-5/2}(27/4)^nn!, 
 	\qquad t = \frac{4\sqrt{3}}{3^{10}\sqrt{\pi}}.
\end{equation}
A formula which will resurface later. 

As a corollary of Theorem \ref{th:graph} the limiting probability that a random labelled planar chordal graph (with the uniform distribution on graphs with $n$ vertices) is connected tends to $p = c/g \approx 0.99963$ as $n\to\infty$. 
In fact it is straightforward to show \cite{GNR13} that the number of connected components is asymptotically distributed as $1 + X$, where $X$ is a Poisson law with parameter $C_0 \approx 0.00037470$, a value computed at the end of Section~\ref{sec:proof1}, so that $p = e^{-C_0}$.

Our second result is about rooted maps. 
A rooted map is a connected planar multigraph with a fixed embedding in the plane in which an edge (the \textit{root edge}) is distinguished and directed. 
Rooted maps where first enumerated by Tutte \cite{T63} and have been since then the object of much study (see \cite{S15} for definitions on maps and an overview on their enumeration). 
We only consider simple maps (those with no loop or multiple edge) since they are the natural objects with respect to the property of being chordal. 

\begin{theorem}\label{th:map}
Let $M_n$ be the number of rooted chordal simple planar maps with $n$ edges, and $B_n$ those which are 2-connected.
Then as $n\to\infty$ we have 
\begin{enumerate}
	\item $B_n  \sim b\cdot n^{-3/2}\cdot \sigma_b^{-n}, 
	\qquad\text{with } b\approx 0.071674 
	\text{ and } \sigma_b^{-1}\approx 3.65370$,

	\item $M_n \sim m\cdot n^{-3/2}\cdot \sigma^{-n}, 
	\qquad\text{with } m\approx 0.12596
	\text{ and } \sigma^{-1}\approx 6.40375$.
\end{enumerate}
\end{theorem}

\noindent
The proof is again based on the structure of 3-connected chordal maps. 
As opposed to the class of general maps, the class of simple chordal maps is again subcritical. 
This is reflected in the term $n^{-3/2}$ instead of the usual $n^{-5/2}$ for classes of planar maps. 
Other natural subcritical classes are outerplanar maps \cite{GN17} and series-parallel maps \cite{DNS21}, but these two classes do not contain 3-connected graphs.

For classes of labelled graphs we use exponential generating functions $\sum g_nx^n/n!$, where $g_n$ is the number of graphs in the class with $n$ vertices. 
For classes of rooted maps we use instead ordinary generating functions $\sum M_n z^n$, where $M_n$ in the number of maps in the class with $n$ edges. 
We say that the variable $x$ marks vertices and $z$ marks edges. 
We use the tools of analytic combinatorics \cite{FS09} to obtain the estimates in Theorems \ref{th:graph} and \ref{th:map}, in particular transfer theorems for obtaining asymptotic estimates of functions with algebraic singularities.
\begin{lemma}\label{lem:transfer}
Assume that $f(z)$ has radius of convergence $\rho > 0$ and admits an analytic continuation to an open domain of the form
\begin{equation*}
	\Delta(R,\phi) = \{ z \colon |z|<R, z \ne \rho, |\arg(z-\rho)| > \phi \}
	\qquad\text{for some } R > \rho \text{ and } 0 < \phi < \pi/2.
\end{equation*}
Further assume that $f(z)$ verifies, when $z\sim \rho$ such that $z\in \Delta(\phi,R)$,
\begin{equation*}
	f(z) \sim c\cdot \left( 1 - \frac{z}{\rho} \right)^{-\alpha}
	\qquad\text{for some } c > 0 \text{ and } \alpha \notin \{0,-1,-2,\dots\}.
\end{equation*}
Then 
the coefficients of $f(z)$ satisfy
\begin{equation*}
	[z^n]f(z) \sim \frac{c}{\Gamma(\alpha)} n^{\alpha-1} \rho^{-n}
	\qquad \hbox{ as } n \to \infty.
\end{equation*}
\end{lemma}

Finally we need the so-called disymmetry theorem for enumerating classes of graphs encoded by trees as presented in \cite{CFKS08}.
A class of graphs ${\cal A}$ is said to be {\em tree-decomposable} if for each graph $\gamma\in {\cal A}$ we can associate in a unique way a tree $\tau(\gamma)$ whose nodes are distinguishable (for instance, by using the labels).
Let ${\cal A}_{\bullet}$ denote the class of graphs in ${\cal A}$ where a node of $\tau(\gamma)$ is distinguished.
Similarly, ${\cal A}_{\bullet - \bullet}$ is the class of graphs in ${\cal A}$ where an edge of $\tau(\gamma)$ is distinguished, and ${\cal A}_{\bullet \rightarrow \bullet}$ those where an edge $\tau(\gamma)$ is distinguished and given a direction.
The dissymmetry theorem  allows us to express the class of unrooted graphs in $\A$  in terms of the rooted classes.

\begin{lemma}\label{lem:species}
	Let ${\cal A}$ be a tree-decomposable class.
	Then
	$$
		{\cal A} + {\cal A}_{\bullet \rightarrow \bullet} \simeq {\cal A}_{\bullet} + {\cal A}_{\bullet - \bullet},
	$$
	where $\simeq$ is a bijection preserving the number of nodes.
\end{lemma}

\noindent
If follows that we can write the series $A_\bullet$ of $\A_\bullet$ in terms of the series of the other classes. 
If the encoding trees have no adjacent nodes of the same type (as in our applications) then ${\cal A}_{\bullet \rightarrow \bullet}$ is twice $ {\cal A}_{\bullet - \bullet}$ and we have  
$$
	A_\bullet(z) = A_\bullet(z) -  {\A}_{\bullet - \bullet}(z).
$$

In Section \ref{sec:gf} we analyse the combinatorial structure of chordal planar graphs according to their connectivity and deduce functional equations satisfied by the associated generating functions. 
Sections \ref{sec:proof1} and \ref{sec:proof2} are devoted to the asymptotic analysis and to the proof of our main results.

%
%
\section{Generating functions of chordal planar graphs}\label{sec:gf}

\subsection{3-connected graphs}

Chordal planar graphs that are 3-connected are in fact chordal triangulations (also called stacked triangulations): they are obtained from $K_4$ by repeatedly adding a vertex in the interior of an existing triangular face and making it adjacent to the three vertices in the boundary. 
This is because chordal graphs admit a perfect elimination ordering, so that when adding a new vertex one has to make it adjacent to exactly three existing vertices in roder to preserve 3-connectedness and planarity. 

Let  $T(z)$ be the generating function of labelled chordal triangulations rooted at a directed edge, where $z$ marks the number of vertices minus 2. 
It is clear that 3-connected chordal maps rooted at a triangle are in bijection with planted ternary trees rooted at a leaf.
The generating function of ternary trees rooted at a leaf, where $z$ marks internal nodes, is given by
\begin{align}\label{eq:S}
	S(z) & = z(1+S(z))^3.
\end{align}
According to the bijection, we have
\begin{equation}\label{eq:T} 
	T(z) = \frac{zS(z)}{2}.
\end{equation}
The division by 2 is because in a map we have two choices for the root face, to the left or to the right of the root edge. 
It is well known that $[z^n]S(z) = \frac{1}{2n+1}\binom{3n}{n}$, from which \eqref{eq:tn} follows. 

Later we need the unrooted version of $T(z)$. 
It can be computed by algebraic integration of $S(z)$ but we choose to use the dissymmetry theorem on ternary trees since the proof is more combinatorial and it is used again later. 

\begin{lemma}\label{lem:integration_U} 
Let $U(z)$ be the generating function counting (unrooted) labelled chordal triangulations, where $z$ marks all  vertices.
Then 
\begin{equation}\label{eq:U}
	U(z) = \frac{z^3}{24}(S(z) - S(z)^2). 
\end{equation}
\end{lemma}

\begin{proof}
Labelled chordal triangulations with $n$ vertices are in bijection with the class of trees with $n-3$ nodes endowed with a labelling that satisfies the following:
\begin{enumerate}
	\item every node is labelled with a subset of size $4$ of $\{1, \dots, n\}$;

	\item the intersection of the label of a node with the labels of its neighbours has size $3$ and the intersection is different for each neighbour. In particular, every node has degree at most $4$;

	\item the graph induced by the nodes whose label contains a given $i\in\{1, \dots, n\}$ is connected.
\end{enumerate}
Indeed, the nodes of the tree correspond to the 4-cliques in the triangulation, their labels correspond to the labels of the four vertices of the 4-clique, and adjacent nodes correspond to two 4-cliques glued through a triangle. 
From this, it is straigthforward to verify that the three properties of the labelling given above are necessary and sufficient.
Therefore, the generating function of the encoding trees, counted by their number of nodes plus 3, is the generating function of labelled chordal triangulations. 
We will use the dissymmetry theorem, i.e. Lemma \ref{lem:species}, on the trees. 
We denote by $A$, $A^{\bullet}$, $A^{\bullet - \bullet}$ and $A^{\bullet \rightarrow \bullet}$ the generating functions of $\mathcal{A}$, $\mathcal{A}^{\bullet}$, $\mathcal{A}^{\bullet - \bullet}$ and $\mathcal{A}^{\bullet \rightarrow \bullet}$, respectively.
Note that since all nodes have different labels they are distinguishable. 
And hence $A^{\bullet \rightarrow \bullet} = 2A^{\bullet - \bullet}$.
    
Rooting a tree at a node corresponds to rooting a chordal triangulation at a 4-clique.
We fix the four vertices of the clique, and then at each triangle we attach a (possibly empty) chordal triangulation.
This gives
\begin{equation*}
    A^{\bullet} = \frac{z^4}{24}\left( 1 + S(z) \right)^4 = \frac{z^3}{24}S(z)\left( 1 + S(z) \right).
\end{equation*}
Rooting at an edge corresponds to rooting a chordal triangulation at a triangle shared by two 4-cliques. 
We fix the three vertices of the clique and then attach two chordal triangulations rooted at it. 
Taking into account symmetries this gives
\begin{equation*}
    A^{\bullet - \bullet} = \frac{z^3}{12}S(z)^2.
\end{equation*}
Finally, we have
\begin{equation*}
    U(z) = A 
    = A^{\bullet} - A^{\bullet - \bullet} 
    = \frac{z^3}{24}S(z)\left( 1 + S(z) \right) - \frac{z^3}{12}S(z)^2 
    = \frac{z^3}{24}(S(z) - S(z)^2).
\end{equation*}
\end{proof}

\subsection{2-connected graphs}

First we consider \emph{networks}, that are 2-connected labelled chordal planar graphs rooted at a directed edge $e$ so that the endpoints of $e$ are not labelled. 
Let $E(x,y)$ be the generating function of networks, where $x$ marks vertices as before and $y$ marks edges. 
The relation between $E(x,y)$ and the corresponding generating function $B(x,y)$ of 2-connected graphs is then
\begin{equation}\label{eq:BEint}
	E(x,y)= \frac{2y}{x^2}B_y(x,y).
\end{equation}

\begin{lemma}\label{lem:decomposition_E} 
The following equation holds:
	\begin{equation}\label{eq:E}
		E(x,y) = y \exp\left(xE(x,y)^2 + \frac{T\left(xE(x,y)^3\right)}{E(x,y)}\right).
	\end{equation}
\end{lemma}

\begin{proof}
Following the classical decomposition \cite{GNR13}, networks are parallel compositions of series compositions and  3-connected components.
A 3-connected component in a network is a chordal triangulation rooted at an edge, in which every edge except the root edge is replaced by a network.
Furthermore, at most two networks can be composed in series, as otherwise there would be an induced cycle of length at least 4.
Thus, a series composition can be encoded as a triangle rooted at an edge, in which each non-root edge is replaced by a network. 
The factor $y$ on the right-hand side of \eqref{eq:E} encodes the root edge, the exponential encodes the set (possibly empty) of parallel netwoks, the term $xE(x,y)^2$ encodes the series composition of exactly two networks, and the term $T(xE^3)/E$  the replacement of edges in a triangulation $t$. 
The term $xE^3$ stems from the fact that we keep the vertices of $t$ and in addtion if $t$ has $n-2$ vertices then it has exactly $3n$ edges; the divsion by $E$ is because the root edge of $t$ is not replaced. 
\end{proof}%

Expanding the first terms of $E$ in series and using \eqref{eq:BEint} we get
\begin{equation*}
	B(x,y)  = \frac{x^2}{2} \int \frac{E(x,y)}{y}dy 
	= \frac{y}{2}x^2 + \frac{y^3}{6}x^3 + \left(\frac{y^6}{24} + \frac{y^5}{4}\right)x^4 
	+ \left(\frac{y^9}{12} + \frac{y^8}{4} + \frac{7\,y^7}{12}\right)x^5 + \cdots
\end{equation*}
We need to compute the integral $\int (E(x,y)/y)dy$ in order to express $B$ in terms of $E$.
It is equivalent to the operation of unrooting a graph, and we do it using again the dissymmetry theorem.
We encode a 2-connected chordal graph with a canonical tree, whose nodes are of the following types: $e$ (edge), $s$ (series), and $t$ (triangulation). 
Notice that the edges in the tree can only be of type $e-s$ or $e-t$.
An example of the decomposition of a 2-connected chordal planar graph and its associated tree is illustrated by Figure \ref{fig:decomposition_tree}.
\begin{figure}
\centering
	\raisebox{20px}{\includegraphics[scale=1]{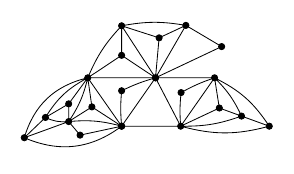}} \qquad
	\includegraphics[scale=1]{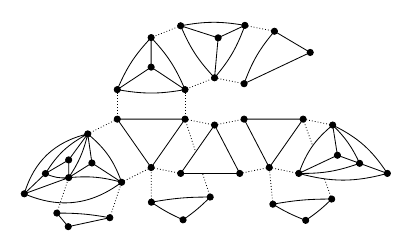}
	\includegraphics[scale=1.2]{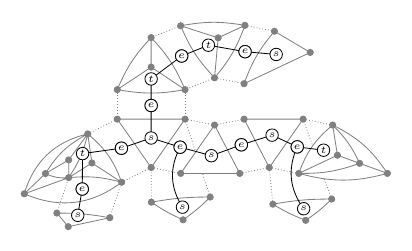}
\caption{
		Top left is an example of a 2-connected chordal planar graph.
		Top right is its decomposition into series and 3-connected components.
		Bottom is the tree associated with the decomposition.
	}
	\label{fig:decomposition_tree}
\end{figure}
We denote by $R_e(x)$, $R_s(x)$, $R_t(x)$, $R_{e-s}(x)$, $R_{e-t}(x)$ the generating functions encoding trees rooted at the corresponding specific type of node or edge. 
By symmetry we have $R_{e-s}(x) = R_{e\to s}(x) = R_{s\to e}(x)$, and the same goes for $R_{e-t}(x)$.

\begin{lemma}
Let $E = E(x) = E(x,1)$ and $B(x) = B(x,1)$.
Then
\begin{equation}\label{eq:BE}
	B(x) = \frac{x^2}{2}\left( E - \frac{xE^3}{12}\left( S\left( xE^3 \right)^2 + 5S\left( xE^3 \right) + 8 \right) \right).
\end{equation}
\end{lemma}

\begin{proof}
Applying Lemma \ref{lem:species} on the tree-decomposable class of labelled 2-connected chordal planar graphs and taking the symmetries into account gives
\begin{equation}\label{eq:B_dissymmetry}
	B(x) = R_s(x) + R_e(x) + R_t(x) - R_{e-s}(x) - R_{e-t}(x).
\end{equation}
Next, we compute each of the series on the right-hand side of \eqref{eq:B_dissymmetry}.

\textit{Rooting at nodes of the tree.}
Rooting the tree at a $s$-node corresponds to fixing three unordered vertices of a triangle and attaching a (possibly empty) network at each of the three edges. 
Thus, $R_s(x) = x^3E^3/6$.
Furthermore, unless the graph is reduced to a single edge encoded by $x^2/2$, nodes of type $e$ have degree at least 2 in the tree. 
Therefore, rooting at an $e$-node corresponds to fixing the two unordered vertices of the edge and attaching a network with at least two parallel components. 
This is encoded by substracting the first two terms of the exponential in \ref{eq:E}, which gives $R_e(x) = \frac{x^2}{2}\left(E - xE^2 - T(xE^3)/E\right)$.
Finally, rooting at a $t$-node corresponds to attaching a network at each of the edges of an unrooted chordal triangulation. 
Considering that $z$ in $U(z)$ encodes the number of vertices minus 2, this gives $R_t(x) = U(xE^3)/E^6$.

\textit{Rooting at edges of the tree.}
Rooting at an ($e-s$)-edge corresponds to rooting a graph at a triangle with one of its sides distinguished. 
Therefore, we fix the three vertices of the triangle (the two vertices on the root edge are unordered) and we attach a non-empty network to the root edge and any network to each of the two remaining edges. 
This yields $R_{e-s}(x) = \frac{x^3}{2} E^2\left(E - 1\right)$.
Finally, rooting at an ($e-t$)-edge corresponds to rooting at a triangulation with a distinguished edge and attaching a network to each of its edges. 
The term $T\left(xE^3\right)$ is explained as before and this gives 
$R_{e-t}(x) = \frac{x^2}{2} T(xE^3)(E - 1)/E$.

Equation \eqref{eq:BE} is finally derived from \eqref{eq:B_dissymmetry}, using \eqref{eq:T} and \eqref{eq:U} with $z = xE^3$, and simplifying. 
\end{proof}

\subsection{Connected graphs and arbitary graphs}

Let $C^\bullet(x,y) = xC_x(x,y)$ be the generating function of vertex rooted chordal connected graphs, where $C_x$ denotes the derivative with respect to $x$. 
Since chordality is preserved under blocks, we have the usual relation between $C^\bullet$ and $B$ (see \cite{GNR13}) enconding the recursive decomposition of a (rooted) connected graph into blocks:
\begin{equation}\label{eq:BC}
	C^\bullet(x,y) = x e^{B'(C^\bullet(x,y),y)}.
\end{equation}
Using the expansion of $B(x,y)$ one gets
\begin{align*}
	C(x,y) = x + y\frac{x^2}{2} + \left(y^3 + 3y^2\right)\frac{x^3}{3!}
	 + \left(y^6 + 6y^5 + 12y^4 + 16y^3\right)\frac{x^4}{4!} \\ 
	 + \left(10y^9 + 30y^8 + 90y^7 + 135y^6 + 150y^5 + 125\,y^4\right)\frac{x^5}{5!} + \cdots 
\end{align*}

Finally the generating function of all chordal planar graphs is $G(x,y) = e^{C(x,y)}$, giving
\begin{align*}
	G(x,y) = 1 + x 
 	+ \left(y + 1\right)\frac{x^2 }{2}
 	+ \left(y^3 + 3y^2 + 3y + 1\right)\frac{x^{3}}{3!}
 	+ \left(y^6 + 6y^5 + 12y^4 + 20y^3 + 15y^2 + 6y + 1\right)\frac{x^4}{4!} \\
 	+ \left(10y^9 + 30y^8 + 90y^7 + 140y^6 + 180y^5 + 195y^4 + y^3 + 45y^2 + 10y + 1\right)\frac{x^5}{5!} + \cdots
\end{align*}
For instance the term $12y^4x^4/4!$ is because there are $15=\binom{6}{4}$ labelled graphs with four vertices and four edges, from which we must remove the three labellings of $C_4$, the smallest non-chordal graph.

%
%
\section{Proof of Theorem \ref{th:graph}}\label{sec:proof1}

\subsection{2-connected graphs}

Using \eqref{eq:T} with $z = x(1 + F)^3$ and setting $y=1$, $F = F(x) = E(x) - 1$ and $S = S(x(1 + F)^3)$, we transform Equations \eqref{eq:S} and \eqref{eq:E} into a system amenable to the so-called Drmota-Lalley-Wood's theorem~\cite[Theorem 2.33]{D09}, with $u=1$, as follows:
\begin{equation}\label{sys:EF}
	\renewcommand{\arraystretch}{1.5}
	\begin{array}{ll}
		F = \exp\left( x(1 + F)^2 + \displaystyle\frac{x(1 + F)^2S}{2} \right) - 1, \\
		S = x(1 + F)^3(1 + S)^3.
	\end{array}
\end{equation}
Let $\Phi(x,S,F)$ and $\Psi(x,S,F)$ be the right hand-side of the first and second equation in \eqref{sys:EF}, respectively.
Those functions are entire and define a system with a strongly connected dependency graph between variables $S$ and $F$.
Furthermore, both have non-negative coefficients and vanish at $x=0$, while they satisfy $\Phi(x,0, 0) \ne 0$ and $\Psi(x,0,0) \ne 0$, but also $\Phi_x(x,S,F) \ne 0$ (where $\Phi_x = \partial \Phi/\partial x$) and $\Psi_x(x,S,F) \ne 0$.
Finally, Sytem~\eqref{sys:EF} extended by its Jacobian admits a solution that is non-zero.
It is given by the following approximations:
\begin{equation}\label{sol:sysE}
	\rho_b\approx 0.092859, 
	\quad E_0 = E(\rho_b) = 1 + F(\rho_b) \approx 1.16454, 
	\quad S_0 = S(\rho_bE_0^3)\approx 0.41919.
\end{equation}
Thus the hypotheses of \cite[Theorem 2.33]{D09} are verified.
This implies in particular that $\rho_b$ is the unique dominant singularity of the function $E(x)$, i.e. on the boundary of the disk of convergence, and that $E(x)$ admits the following analytic continuation in a domain of the form $\Delta(R_b,\phi_b)$ for some $R_b > \rho_b$ and $0 < \phi_b < \pi/2$:
\begin{equation}\label{eq:puisE}
	E(x) = E_0 - E_1\sqrt{1 - \frac{x}{\rho_b}} + O\left( 1 - \frac{x}{\rho_b} \right)
	\qquad\text{for } x\sim\rho_b \text{ and  } x\in\Delta(R_b,\phi_b),
\end{equation}
where $E_1 > 0$ is computed next.
Since $S$ is itself a function of $x$ and $F$, \eqref{sys:EF} can be re-written as $F = \Theta(x,F)$, where $\Theta$ is analytic at $(\rho_b,F_0)$ with $F_0 = F(\rho_b)$. 
One checks that $\Theta_x(\rho_b,F_0)\neq 0$, $\Theta_F(\rho_b,F_0) = 0$ and $\Theta_{FF}(\rho_b,F_0)\neq 0$.
And we can apply \cite[Lemma VII.3]{FS09} (see also \cite[Remark 2.20]{D09}) to obtain
\begin{equation}\label{eq:puisE_coeff}
	E_1 = F_1 = \sqrt{\frac{2\rho_b \Theta_x(\rho_b,F_0)}{\Theta_{FF}(\rho_b,F_0)}}
	\approx 0.092354.
\end{equation}
Furthermore, \cite[Theorem 2.33]{D09} implies a similar results for $S = S(x(1 + F)^3)$.
Note also that $\rho_bE_0^3 = 0.14665 < 4/27$, where $4/27$ is the dominant singularity of $S(z)$.
This implies thar the composition scheme $S(xE(x)^3)$ is subcritical in the sense of \eqref{eq:subcritical}. 

With those results at hand, we can finally consider the generating function $B(x)$.
Given the expression~\eqref{eq:BE} and the fact that the scheme is subcritical, the dominant singularity of $B(x)$ is the same as that of $E(x)$ and it is furthermore unique.
We show next that $B(x)$ admits a singular expansion at $z=\rho_b$ similar to $E(x)$.
First we extend the system \eqref{sys:EF} so that it includes the variable $y$:
\begin{equation}\label{sys:EFy}
	\renewcommand{\arraystretch}{1.5}
	\begin{array}{ll}
		F = y\exp\left( x(1 + F)^2 + \displaystyle\frac{x(1 + F)^2S}{2} \right) - 1, \\
		S = x(1 + F)^3(1 + S)^3,
	\end{array}
\end{equation}
where now $F = F(x,y)$.
By \cite[Theorem 2.33]{D09} (setting $u=y$) there exist three functions $\rho_b(y)$, $f_0(y)$ and $f_1(y)$ analytic in a neighbourhood $W$ of 1 such that for $y\in W$ and $x\sim\rho_b(y)$ with $|\text{arg}(x - \rho_b(y))|\ne 0$ the following singular expansion holds
\begin{equation}\label{eq:puisExy}
	E(x,y) = 1 + F(x,y) = 1 + f_0(y) - f_1(y)\sqrt{1 - \frac{x}{\rho_b(y)}} + O\left(1 - \frac{x}{\rho_b(y)}\right),
\end{equation}
where $\rho_b(1) = \rho_b$, $1 + f_0(1) = E_0$ and $f_1(1) = E_1$.
From there, applying \cite[Theorem 2.30]{D09} to \eqref{eq:BEint} and setting $y=1$, we obtain that $B(x) = B(x,1)$ admits an analytic continuation of the form
\begin{equation*}
	B(x) = B_0 - B_2\left(1 - \frac{x}{\rho_b}\right) + B_3\left(1 - \frac{x}{\rho_b}\right)^{3/2}
	+ O\left(1 - \frac{x}{\rho_b}\right)^2
	\qquad\text{for } x\sim\rho_b \text{ and  } x\in\Delta(R_b,\phi_b).
\end{equation*}
The above coefficients can be computed by substituting into \eqref{eq:BE} the expansions of $E(x)$ and $S(xE(x)^3)$ when $x=\rho_b(1 - X^2)$, with $X = \sqrt{1 - x/\rho_b}$.
This gives $B_0\approx 0.0044796$, $B_2\approx 0.0085328$ and $B_3\approx 0.00038321$.
The estimate on $b_n$ follows from Lemma \ref{lem:transfer}, with $b = 3B_3/(4\sqrt{\pi})\approx 0.00016215$.
\qed

\subsection{Connected and arbitrary graphs}

The composition scheme \eqref{eq:BC} is subcritical because $B''(\rho_b)\to\infty$ (see \cite{GNR13}).
This means in particular that the singularities of $C^\bullet(x)$ come from a branch point and not from those of $B(x)$ and are obtained by solving 
\begin{align*} 
	\rho = \tau e^{-B'(\tau)}
	\quad\text{and}\quad
	\tau B''(\tau) = 1,
	\qquad\qquad\text{ with } \tau = C^\bullet(\rho) < \rho_b.
\end{align*} 
To find such a solution, one must first compute $E'(x)$ and $E''(x)$ and then $B''(x)$. 
This is a routine but lengthy computation, best solved numerically\footnote{We used \texttt{Maple 2021} for those computations.} 
together with equations \eqref{eq:S} and \eqref{eq:E}, and which gives the following approximate solutions:
\begin{equation}\label{sol:tau}
	\tau\approx 0.092859
	\quad\text{and}\quad 
	E(\tau)\approx 1.16446.
\end{equation}
So that we obtain a singularity of $C^\bullet(x)$ at $x=\rho$ given by 
\begin{align*}
	\rho = \tau e^{-B'(\tau)}\approx 0.084088.
\end{align*}

As before $C^{\bullet}(x)$ can be extended analytically to a domain of the form $\Delta(R,\phi)$ for some $R > \rho$ and $0 < \phi < \pi/2$. 
The same holds for $C(x)$ (see \cite[Proposition 3.10.(1)]{GNR13}), which in fact verifies
\begin{align*}
	C(x) = C_0 - C_2\left(1 - \frac{x}{\rho}\right) + C_3\left(1 - \frac{x}{\rho}\right)^{3/2} + O\left(1 - \frac{x}{\rho}\right)^2
	\qquad\text{for } x\sim\rho \text{ and  } x\in\Delta(R,\phi).
\end{align*}
The above coefficients are given by:
\begin{align*}
	C_0 = \tau (1+ \log \rho - \log \tau) + B(\tau) \approx 0.00037470,
	\qquad C_2 = \tau \approx 0.092859 \\
	\text{and } C_3 = \frac{3}{2}\sqrt{\frac{2\rho \exp(B'(\tau))}{\tau B'''(\tau)-\tau B''(\tau)^2+2B''(\tau)}}\approx 0.00027194.
\end{align*}
The estimate for $c_n$ is again a consequence of Lemma \ref{lem:transfer}.
The same goes for the series  $G(x,y) = e^{C(x,y)}$ of arbitrary chordal planar graphs.
Since $G(x) = e^{C_0}(1 - C_2(1 - x/\rho ) + C_3(1 - x/\rho )^{3/2} + O(1 - x/\rho )^2)$ for $x\sim\rho$ and $x\in\Delta(\phi,R)$, we have
\begin{align*}
	& G_0 = e^{C_0} \approx 1.00037, &
	& G_2 = C_2e^{C_0} \approx 0.092894, &
	& G_3 = C_3e^{C_0} \approx 0.00027205,
\end{align*}
and the estimate for $g_n$ follows. 
This concludes the proof of Theorem \ref{th:graph}.
\qed

%
%
\section{Simple chordal planar maps}\label{sec:proof2}

\paragraph{Decomposition of 2-connected simple chordal maps.}

Let $D(z)$ be the generating function of simple 2-connected chordal maps, where $z$ marks the number of edges minus 1, and let $S(z)$ be the generating function of ternary trees satisfying \eqref{eq:S}.
Similarly to the case of graphs, a simple 2-connected chordal map can be decomposed into a \emph{sequence} of smaller chordal maps. 
As opposed to the situation for graphs the planar embedding provides a linear ordering and we use the sequence instead of the set construction. 
The maps in the sequence are either a triangle rooted at an edge, where each side of the two non-root edges (four sides in total) is replaced by a map, or 3-connected maps in which the two sides of every edge are replaced by a map. 
This gives
\begin{equation}\label{eq:maps_D}
	D(z) = \frac{1}{1 - z^2D(z)^4\left(1 + S\left(z^3D(z)^6\right)\right)}.
\end{equation}
This implicit equation determines $D(z)$ uniquely as a series with non-negative coefficients.

\paragraph{Proof of Theorem \ref{th:map}, item 1.}

Let $B(z)$ be the generating function counting simple 2-connected chordal maps, with $z$ now marking the total number of edges, so that $B = B(z)$.
Algebraic elimination between \eqref{eq:S} and \eqref{eq:maps_D} gives the following irreducible polynomial equation satisfied by $B = B(z)$:
\begin{equation}\label{eq:maps_B_poly}
	B^9 - z^2B^5 + z^3B^4 + z^3B^3 - 3z^4B^2 + 3z^5B - z^6 = 0.
\end{equation}
Therefore, $B(z)$ is an algebraic function and its analysis in the rest of the proof will follow the approach detailed in \cite[Chapter VII.7]{FS09}.
For instance, $B(z)$ can be represented at $z=0$ as a Taylor series with non-negative coefficients and radius of convergence $\sigma_b$, for some $\sigma_b > 0$, corresponding to a branch of the curve~\eqref{eq:maps_B_poly} passing through the origin, as follows:
\begin{align*}
	B(z) = z + z^3 + 5z^5 + z^6 + 35z^7 + 16z^8 + 288z^9 + O(z^{10}).
\end{align*}

Next, we find the value of $\sigma_b$.
By Pringsheim's theorem (see \cite[Theorem IV.6]{FS09}), it must be a singularity of $B(z)$.
Since $B(z)$ is algebraic, its singularities must be among the roots of the discriminant of \eqref{eq:maps_B_poly} with respect to $B$, which up to a trivially non-zero factor is equal to 
\begin{equation*}
	387420489z^6 + 573956280z^5 + 184705272z^4 - 81168524z^3 
	- 15907392z^2 + 3326272z - 135424.
\end{equation*} 
This polynomial admits $\sigma_b \approx 0.27370$ as unique positive real root and it can be readily checked that no other root $\psi$ satisfies $|\psi| = \sigma_b$.

Finally, we determine the singular expansion of $B(z)$ locally around $\sigma_b$.
As $B(z)$ is algebraic and has no other singularity on the circle of radius $\sigma_b$, there exists $R_b' > \sigma_b$ and $0 < \phi_b' < \pi/2$ for which its representation at $z=0$ admits an analytic continuation to a domain at $z = \sigma_b$ of the form $\Delta(R_b',\phi_b')$.
It can in fact be computed from \eqref{eq:maps_B_poly} using Newton's \textit{polygon algorithm}.
This gives a singular expansion of the form:
\begin{equation}\label{eq:maps_puis_B}
	B(z) = B(\sigma_b) + b_1\sqrt{1 - \frac{z}{\sigma_b}} + O\left(1 - \frac{z}{\sigma_b}\right),
	\qquad\text{for } z\sim\sigma_b \text{ and  } z\in\Delta(R_b',\phi_b'),
\end{equation} 
where $B(\sigma_b)\approx 0.33301$ and $b_1\approx 0.12704$.
The estimate on $B_n$ then follows from Lemma \ref{lem:transfer}.
\qed

\paragraph{Decomposition of simple chordal maps.}

Let  $M(z)$ be the generating function of all simple chordal maps, where  $z$ marks the total number of edges. 
The decomposition of a map into block is given by the equation 	
\begin{equation}\label{eq:maps_M}
	M(z) = B\left(z(1 + M(z))^2\right),
\end{equation}
reflecting the fact that a map is obtained from its 2-connected core by attaching a (possibly empty) map at each corner \cite{T63}. 
Since being simple and chordal is preserved by taking 2-connected components, the same equation holds for simple chordal maps.

\paragraph{Proof of Theorem \ref{th:map}, item 2.}

We proceed as in the proof of item 1.
First, by algebraic elimination between \eqref{eq:maps_B_poly} and \eqref{eq:maps_M}, we obtain an irreducible polynomial equation satisfied by $M = M(z)$:
\begin{equation}\label{eq:maps_M_poly}
	\begin{split}
		& z^6M^{12} + 3z^5\left(4z - 1\right)M^{11} 
		+ z^3\left(66z^3 - 30z^2 + 3z - 1\right)M^{10} \\
		& + \left(220z^6 - 135z^5 + 24z^4 - 7z^3 + z^2 - 1\right)M^9  
		+ z^2\left(495z^4 - 360z^3 + 84z^2 - 21z + 4\right)M^8 \\
		& + z^2\left(792z^4 - 630z^3 + 168z^2 - 35z + 6\right)M^7 
		+ z^2\left(924z^4 - 756z^3 + 210z^2 - 35z + 4\right)M^6 \\
		& + z^2\left(792z^4 - 630z^3 + 168z^2 - 21z + 1\right)M^5
		+ \left(495z^6 - 360z^5 + 84z^4 - 7z^3\right)M^4 \\
		& + z^3\left(220z^3 - 135z^2 + 24z - 1\right)M^3
		+ 3z^4\left(22z^2 - 10z + 1\right)M^2
		+ 3z^5\left(4z - 1\right)M
		+ z^6 = 0,
	\end{split}
\end{equation}
From the curve \eqref{eq:maps_M_poly} we get that $M(z)$ can be represented at $z=0$ as the following Taylor series with non-negative coefficients and radius of convergence $\sigma > 0$:
\begin{align*}
	M(z) = z + 2z^2 + 6z^3 + 22z^4 + 92z^5 + 419z^6 + 2025z^7 
	+ 10214z^8 + 53192z^9 + O(z^{10}).
\end{align*}
The discriminant of \eqref{eq:maps_M_poly} with respect to $M$ is, up to a trivially non-zero factor, given by
\begin{equation}\label{eq:maps_poly_sing_M}
	\begin{split}
		& 2035256037376z^{12} - 2215690119168z^{11} + 6474387490048z^{10} 
		+ 1262789263168z^9 \\ 
		& - 3620212090976z^8 + 1275725763644z^7 - 301902286683z^6 + 60575733276z^5 \\
		& - 13112588384z^4 - 5212588972z^3 + 1812419712z^2 - 148471488z + 3656448.
	\end{split}
\end{equation}
It admits two positive real roots, given approximately by $0.15616$ and $0.49512$. 
However $0.49512$ cannot be the radius of convergence of $M(z)$ since it is larger than $\sigma_b$.
Therefore $\sigma\approx 0.15616$, and it can be readily checked that \eqref{eq:maps_poly_sing_M} admits no other zero of modulus $\sigma$.
This means that there exists $R' > \sigma$ and $0 < \phi' < \pi/2$ for which the representation of $M(z)$ at $z=0$ admits an analytic continuation to a domain at $z = \sigma$ of the form $\Delta(R',\phi')$.
It is given by
\begin{equation}\label{eq:maps_puis_M}
	M(z) = M(\sigma) + m_1\sqrt{1 - \frac{z}{\sigma}} + O\left(1 - \frac{z}{\sigma}\right),
	\qquad\text{for } z\sim\sigma \text{ and  } z\in\Delta(R',\phi'),
\end{equation} 
where $M(\sigma)\approx 0.31055$ and $m_1\approx 0.22326$.
Note that the class of simple chordal maps is subcritical in the sense, similar to \eqref{eq:subcritical}, that the composition scheme in \eqref{eq:maps_M} is subcritical, that is, $\sigma(1 + M(\sigma)^2)\approx 0.26821 < \sigma_b$.
The estimate on $M_n$ is obtained from Lemma \ref{lem:transfer} as before, and this concludes the proof of Theorem \ref{th:map}. 
\qed

%
%
\section{Concluding remarks}

From the system \eqref{sys:EFy} and \cite[Theorem 2.35]{D09} we could obtain, applying the so-called `quasi-powers theorem' \cite{FS09}, a central limit theorem for the number of edges in a uniform random 2-connected chordal planar graph with $n$ vertices as $n\to\infty$.
This result is to be expectd and fits into a general scheme of similar Gaussian parameters in subcritical graph classes (see for instance \cite{DFKKR11}, and \cite{DRR17} and \cite{RRW20} for some generalisations).
It would be of interest to study in the context of chordal planar graphs other parameters, particularly extremal parameters \cite{GNR13}.

Furtermore, by sligthly adapting the scheme developed in this paper, one could in principle enumerate several related families of chordal graphs, such as outerplanar and series-parallel graphs,  planar multigraphs and non-necessarily simple planar maps. 
But also non-planar graphs, such as forbidding $K_{3,3}$ or $K_5$ as a minor.
For chordal graphs, forbidding $K_5$ as a minor is equivalent to the property of having tree-width at most three.  
A future line of research is to enumerate chordal graphs with bounded tree-width.

To conclude we display the first numbers of labelled chordal planar graphs (resp. maps) counted by vertices in Table \ref{tab:graphs} (resp. counted by edges in Table \ref{tab:maps}) for the different families studied in this work.
\begin{table}[htb]
\centering
\scriptsize
\begin{tabular}{crrr}
\hline
$n$ & $g_n$ & $c_n$ & $b_n$\\
\hline
1 & 1 & 1 & 0 \\
2 & 2 & 1 & 1 \\
3 & 8 & 4 & 1 \\
4 & 61 & 35 & 7 \\
5 & 821 & 540 & 110 \\
6 & 17962 & 13116 & 2880 \\
7 & 589912 & 462868 & 108486 \\
8 & 26990539 & 22189056 & 5376448 \\
9 & 1611421595 & 1364476032 & 330554736 \\
10 & 119106036226 & 102768330140 & 24223100940 \\
11 & 10475032926304 & 9150009283316 & 2056900853260 \\
12 & 1064759262580675 & 937871756182824 & 198279609266376 \\
13 & 122455558249650523 & 108501459033647056 & 21365210239261824 \\
14 & 15683814373288014514 & 13957140054455406368 & 2542622031178234096 \\
15 & 2210104382919809469776 & 1973316500054545453200 & 331005569819483825280 \\
16 & 339419270505312015418873 & 303844760227083629476736 & 46769563108388612386560 \\
17 & 56377137858208036652271961 & 50574398535605806604877952 & 7125735843407702680130176 \\
18 & 10064213826097447392585326650 & 9043978529936559892024953936 & 1164214191212133452455716432 \\
19 & 1920763688236792486611031950040 & 1728560464917767130397726200016 & 203006967721530831955744610256 \\
20 & 390147921384971528200998632189581 & 351542184165686400289151814740320 & 37624686779731200180043318035040 \\
\hline
\end{tabular}
\caption{Numbers of arbitrary, connected and 2-connected labelled chordal planar graphs with $n$ vertices.}
\label{tab:graphs}
\end{table}
\begin{table}[htb]
\centering
\scriptsize
\begin{tabular}{crrr}
\hline
$n$ & $M_n$ & $B_n$ \\
\hline
1 & 1 & 1 \\
2 & 2 & 0 \\
3 & 6 & 1 \\
4 & 22 & 0 \\
5 & 92 & 5 \\
6 & 419 & 1 \\
7 & 2025 & 35 \\
8 & 10214 & 16 \\
9 & 53192 & 288 \\
10 & 283921 & 210 \\
11 & 1545326 & 2607 \\
12 & 8544766 & 2612 \\
13 & 47867107 & 25155 \\
14 & 271091848 & 31885 \\
15 & 1549624321 & 254255 \\
16 & 8929009486 & 386672 \\
17 & 51807558686 & 2663101 \\
18 & 302430309885 & 4682253 \\
19 & 1774979731304 & 28696460 \\
20 & 10467456794046 & 56747900 \\
\hline
\end{tabular}
\caption{Numbers of arbitrary and 2-connected simple chordal maps with $n$ edges.}
\label{tab:maps}
\end{table}

\subsection*{Acknowledgements}

We gratefully acknowledge earlier discussions on this project with Erkan Narmanli.
M.N. was supported by grants MTM2017-82166-P and PID2020-113082GB-I00, the Severo Ochoa and María de Maeztu Program for Centers and Units of Excellence in R\&D (CEX2020-001084-M).
C.R. was supported by the grant Beatriu de Pin\'os BP2019, funded by the H2020 COFUND project No 801370 and AGAUR (the Catalan agency for management of university and research grants), and the grant PID2020-113082GB-I00 of the Spanish ministry of Science and Innovation.

\bibliographystyle{abbrv}
\bibliography{biblio_chordal_planar}

\end{document}